\documentclass{amsart}

\usepackage{amsmath,amsthm,amsfonts,amssymb,amscd,latexsym}
\usepackage{color}

\newtheorem{thm}{Theorem}[section]
\newtheorem{lem}[thm]{Lemma}
\newtheorem{pro}[thm]{Proposition}
\newtheorem{cor}[thm]{Corollary}

\numberwithin{equation}{section}

\newcommand{\ch}{\mathrm{char}}

\newcommand{\ad}{\mathrm{ad}}
\newcommand{\id}{\mathrm{id}}
\newcommand{\Hom}{\mathrm{Hom}}
\newcommand{\End}{\mathrm{End}}
\newcommand{\Ker}{\mathrm{Ker}}
\newcommand{\im}{\mathrm{Im}}

\newcommand{\ann}{\mathrm{Ann}}

\newcommand{\leib}{\mathrm{Leib}}
\newcommand{\lie}{\mathrm{Lie}}
\newcommand{\sym}{\mathrm{sym}}
\newcommand{\HCE}{\mathrm{H}}
\newcommand{\CL}{\mathrm{CL}}
\newcommand{\dl}{\mathrm{d}}
\newcommand{\ZL}{\mathrm{ZL}}
\newcommand{\BL}{\mathrm{BL}}
\newcommand{\HL}{\mathrm{HL}}
\newcommand{\DL}{\mathrm{DL}}
\newcommand{\res}{\mathrm{res}}

\newcommand{\N}{\mathbb{N}}
\newcommand{\F}{\mathbb{F}}

\newcommand{\lf}{\mathfrak{L}}
\newcommand{\iif}{\mathfrak{I}}
\newcommand{\nf}{\mathfrak{N}}
\newcommand{\af}{\mathfrak{A}}
\newcommand{\kf}{\mathfrak{K}}
\newcommand{\hf}{\mathfrak{H}}
\newcommand{\mf}{\mathfrak{M}}
\newcommand{\gl}{\mathfrak{gl}}
\newcommand{\gf}{\mathfrak{g}}
\newcommand{\rf}{\mathfrak{R}}

\begin{document}


\title[Cohomology of solvable Leibniz algebras]{On the cohomology of solvable Leibniz algebras}

\author{J\"org Feldvoss}
\address{Department of Mathematics and Statistics, University of South Alabama,
Mobile, AL 36688-0002, USA}
\email{jfeldvoss@southalabama.edu}

\author{Friedrich Wagemann}
\address{Laboratoire de math\'ematiques Jean Leray, UMR 6629 du CNRS,
Universit\'e de Nantes, 2, rue de la Houssini\`ere, F-44322 Nantes Cedex 3,
France}
\email{wagemann@math.univ-nantes.fr}

\subjclass[2020]{Primary 17A32; Secondary 17B56}

\keywords{Leibniz cohomology, Chevalley-Eilenberg cohomology, Hochschild
cohomology, cohomological (non-)vanishing, semi-simple Leibniz algebra,
nilpotent Leibniz algebra, inner derivation, infinitesimal deformation, rigid
Leibniz algebra, supersolvable Leibniz algebra, solvable Leibniz algebra,
right faithful Leibniz bimodule, maximal subalgebra, Frattini subalgebra,
minimal ideal, right self-centralizing ideal}


\begin{abstract}
This paper is a sequel to \cite{FW}, where we mainly consider semi-simple
Leibniz algebras. It turns out that the analogue of the Hochschild-Serre spectral
sequence for Leibniz cohomology cannot be applied to many ideals, and therefore
this spectral sequence seems not to be applicable for computing the cohomology
of non-semi-simple Leibniz algebras. The main idea of the present paper is to
use similar tools as developed by Farnsteiner for Hochschild cohomology (see
\cite{F1} and \cite{F2}) to work around this. Unfortunately, it does not seem
to be possible to relate the cohomology of a Leibniz algebra directly to Hochschild
cohomology as is the case for Lie algebras, but all the desired results can be
obtained in a similar way. In particular, this enables us to generalize the vanishing
theorems of Dixmier and Barnes for nilpotent and (super)solvable Lie algebras
to Leibniz algebras. Moreover, we compute the cohomology of the one-dimensional
Lie algebra with values in an arbitrary Leibniz bimodule and show that it is periodic
with period two. As a consequence, we prove the Leibniz analogue of a non-vanishing
theorem of Dixmier. Although not needed in full for the aforementioned results,
we prove a Fitting lemma for Leibniz bimodules that might be useful elsewhere.
\end{abstract}


\date{April 4, 2023}
          
\maketitle


\section*{Introduction}


In a previous paper \cite{FW} we started to study the cohomology of (left) Leibniz
algebras. One of our main results is the second Whitehead lemma for finite-dimensional
semi-simple Leibniz algebras in characteristic zero (see \cite[Theorem 4.3]{FW}). More
generally, we systematically adapted Pirashvili's spectral sequences (see \cite{P}) to
cohomology and general Leibniz bimodules. One of these spectral sequences (see
\cite[Theorem 3.4 or Corollary 3.5]{FW}) is the Leibniz analogue of the Hochschild-Serre
spectral sequence. It is clear from the $E_2$-term that this spectral sequence is not
useful for the computation of the cohomology of Leibniz algebras with abelian ideals
different from the Leibniz kernel. 

In this paper we consider methods developed in order to prove vanishing theorems for
Hochschild cohomology. This enables us to extend several well-known vanishing theorems
from the cohomology of solvable Lie algebras to Leibniz algebras. In fact, in \cite{F1}
and \cite{F2}, Farnsteiner's aim is (among other results) to unify the proofs of Dixmier's
vanishing theorem for nilpotent Lie algebras (see \cite[Th\'eor\`eme 1]{D}) and Barnes'
vanishing theorems for (super)solvable Lie algebras (see \cite[Theorems~2 and 3]{B1}).
The crucial idea to prove the vanishing of cohomology is to employ a generalization of
Casimir elements. The latter were used by Whitehead to prove the vanishing of the
cohomology of semi-simple Lie algebras in characteristic zero. Farnsteiner showed
that similar ideas can be employed to prove cohomological vanishing theorems for Lie
algebras with non-zero abelian ideals and Lie algebras in prime characteristic. More
precisely, one needs the existence of certain elements that act invertibly on the
cohomology, while the algebra acts trivially on its cohomology due to certain Cartan
type relations. The conclusion is then that the cohomology must vanish. Since it does
not seem to be possible to express Leibniz cohomology in terms of Hochschild cohomology,
Farnsteiner's results cannot be applied directly, but they have to be adapted to Leibniz
algebras. It turns out that this is possible (see Section~3). In the first section we recall
some definitions and prove several basic results that will be useful later in the paper.
Section 2 is devoted to the Fitting decomposition for Leibniz bimodules. Note that in
the proof of one of the main results in Section~3, namely, Theorem \ref{fittinghh},
we only need that the Fitting-0-component of a Leibniz bimodule is a sub-bimodule.
Since this might be useful elsewhere, we also prove the analogue of Fitting's lemma
for Leibniz algebras.

The main results of the present paper are contained in Section 4. In 1955, Dixmier
\cite{D} proved (non-)vanishing theorems for the Chevalley-Eilenberg cohomology
of finite-dimensional nilpotent Lie algebras. The vanishing behavior depends on the
coefficients having or not having a trivial composition factor. Later, Barnes \cite{B1}
gave a different proof of Dixmier's vanishing theorem (see the proof of \cite[Lemma~3]{B1})
using the Hochschild-Serre spectral sequence and induction on the dimension of the
Lie algebra. On the other hand, Dixmier's proof of his non-vanishing theorem (see
\cite[Th\'eor\`eme 2]{D}) relies on a long exact sequence related to the kernel
of the restriction map from the cohomology of the nilpotent Lie algebra to the
cohomology of an ideal of codimension one. Our overall goal in the present paper
is to generalize these results to (left) Leibniz algebras. We prove analogues of
Dixmier's vanishing theorem for nilpotent Lie algebras (see Theorem \ref{dixmier})
and Barnes' vanishing theorems for (super)solvable Lie algebras (see Theorems
\ref{barnes} and  \ref{vansolv}). Another part of Section 4 is then devoted to
establish a Leibniz analogue of Dixmier's non-vanishing theorem for nilpotent Lie
algebras. We proceed as close as possible to Dixmier's proof, but there are
several obstacles. The base step follows from the computation of the cohomology
of the one-dimensional Leibniz algebra with values in an arbitrary Leibniz bimodule.
Similar to the cohomology of finite cyclic groups, in positive degrees this cohomology
is periodic with period 2 (see Theorem~\ref{ab}). This allows us then to proceed
by induction on the dimension of the nilpotent Leibniz algebra, but our cohomological
non-vanishing theorem (see Theorem \ref{nonvannilp}) is weaker than what one
would expect from Dixmier's result (see Proposition \ref{nontriv} and Examples A
and C). Nevertheless, a consequence of Theorem~\ref{nonvannilp} is that the
cohomology of a nilpotent Leibniz algebra with trivial coefficients does not vanish
in any degree (see Corollary \ref{nonvantriv}). Moreover, the sufficient condition
for the non-vanishing of the adjoint cohomology of a nilpotent Leibniz algebra in
every degree (see Corollary \ref{adj}), which one obtains as a special case of
Theorem \ref{nonvannilp}, is easy to verify, and it is always satisfied for a nilpotent
Lie algebra (see Corollary \ref{adjlie}).

As an application of the vanishing theorems in Section 4 we extend some
structure theorems for (super)solvable Lie algebras (see \cite[Section 3]{B1})
to Leibniz algebras.

In this paper we will follow the notation used in \cite{F} and \cite{FW}. An algebra
without any specification will be a vector space with a bilinear multiplication that
not necessarily satisfies any other identity. Ideals will always be two-sided ideals
if not explicitly stated otherwise. All vector spaces
and algebras are defined over an arbitrary field which is only explicitly mentioned
when some additional assumptions on the ground field are made or this enhances
the understanding of the reader. In particular, all tensor products are over the
relevant ground field and will be denoted by $\otimes$. For a subset $X$ of a
vector space $V$ over a field $\F$ we let $\langle X\rangle_\F$ be the subspace
of $V$ spanned by $X$. We will denote the space of linear transformations from
an $\F$-vector space $V$ to an $\F$-vector space $W$ by $\Hom_\F(V,W)$. In
particular, $\End_\F(V):=\Hom_\F(V,V)$ is the space of linear operators on $V$,
and $V^*:=\Hom_\F(V,\F)$ is the space of linear forms on $V$. Finally, the identity
function on a set $X$ will be denoted by $\id_X$, the set $\{1,2,\dots\}$ of positive
integers will be denoted by $\N$, and the set $\N\cup\{0\}$ of non-negative integers
will be denoted by $\N_0$.


\section{Preliminaries}\label{prelim}


In this section we recall some definitions and collect several results that will be useful
in the remainder of the paper.

A {\em left Leibniz algebra\/} is an algebra $\lf$ such that every left multiplication
operator $L_x:\lf\to\lf$, $y\mapsto xy$ is a derivation. This is equivalent to the identity
$$x(yz)=(xy)z+y(xz)$$
for all $x,y,z\in\lf$, which in turn is equivalent to the identity
$$(xy)z=x(yz)-y(xz)$$
for all $x,y,z\in\lf$. There is a similar definition of a {\em right Leibniz algebra\/}, but
in this paper we will only consider left Leibniz algebras which often will
just be called Leibniz algebras unless this might make matters easier to understand
for the reader. 

Note that every Lie algebra is a left and a right Leibniz algebra. On the other hand, every
Leibniz algebra has an important ideal, its Leibniz kernel, that measures how much
the Leibniz algebra deviates from being a Lie algebra. Namely, let $\lf$ be a Leibniz
algebra over a field $\F$. Then $$\leib(\lf):=\langle x^2\mid x\in\lf\rangle_\mathbb{F}$$
is called the {\em Leibniz kernel\/} of $\lf$. The Leibniz kernel $\leib(\lf)$ is an abelian
ideal of $\lf$, and $\leib(\lf)\ne\lf$ whenever $\lf\ne 0$ (see \cite[Proposition 2.20]{F}).
Moreover, $\lf$ is a Lie algebra if, and only if, $\leib(\lf)=0$.

By definition of the Leibniz kernel, $\lf_\lie:=\lf/\leib(\lf)$ is a Lie algebra which we call
the {\em canonical Lie algebra\/} associated to $\lf$. In fact, the Leibniz kernel is the
smallest ideal such that the corresponding factor algebra is a Lie algebra (see
\cite[Proposition 2.22]{F}).

Next, we will briefly discuss left modules and bimodules of left Leibniz algebras. Let
$\lf$ be a left Leibniz algebra over a field $\F$. A {\em left $\lf$-module\/} is a vector
space $M$ over $\F$ with an $\F$-bilinear left $\lf$-action $\lf\times M\to M$, $(x,m)
\mapsto x\cdot m$ such that $$(xy)\cdot m=x\cdot(y\cdot m)-y\cdot(x\cdot m)$$ is
satisfied for every $m\in M$ and all $x,y\in\lf$.

Moreover, every left $\lf$-module $M$ gives rise to a homomorphism $\lambda:\lf\to
\gl(M)$ of left Leibniz algebras, defined by $\lambda_x(m):=x\cdot m$, and vice versa.
We call $\lambda$ the {\em left representation\/} of $\lf$ associated to $M$.

The correct concept of a module for a left Leibniz algebra $\lf$ is the notion of a Leibniz
bimodule (see \cite[Section 3]{F} for the motivation behind this definition of a bimodule
for a left Leibniz algebra). An {\em $\lf$-bimodule\/} is a vector space $M$ with an
$\F$-bilinear left $\lf$-action and an $\F$-bilinear right $\lf$-action that satisfy the
following compatibility conditions:
\begin{enumerate}
\item[(LLM)] \hspace{2.5cm}$(xy)\cdot m=x\cdot(y\cdot m)-y\cdot(x\cdot m)$
\item[(LML)] \hspace{2.5cm}$(x\cdot m)\cdot y=x\cdot(m\cdot y)-m\cdot(xy)$
\item[(MLL)] \hspace{2.5cm}$(m\cdot x)\cdot y=m\cdot(xy)-x\cdot(m\cdot y)$
\end{enumerate}
for every $m\in M$ and all $x,y\in\lf$.

It is an immediate consequence of (LLM) that every Leibniz bimodule is a left Leibniz
module.

On the other hand, a pair $(\lambda,\rho)$ of linear transformations $\lambda:\lf\to
\End_\F(V)$ and $\rho:\lf\to\End_\F(V)$ is called a {\em representation\/} of $\lf$
on the $\F$-vector space $V$ if the following conditions are satisfied:
\begin{equation}\label{LLMrep}
\lambda_{xy}=\lambda_x\circ\lambda_y-\lambda_y\circ\lambda_x
\end{equation}
\begin{equation}\label{LMLrep}
\rho_{xy}=\lambda_x\circ\rho_y-\rho_y\circ\lambda_x
\end{equation}
\begin{equation}\label{MLLrep}
\rho_y\circ\rho_x=-\rho_y\circ\lambda_x
\end{equation}
for any elements $x,y\in\lf$. Note that (LML) and (MLL) are equivalent to (\ref{LMLrep})
and (\ref{MLLrep}).

Then every $\lf$-bimodule $M$ gives rise to a representation $(\lambda,\rho)$ of $\lf$
on $M$ via $\lambda_x(m):=x\cdot m$ and $\rho_x(m):=m\cdot x$. Conversely, every
representation $(\lambda,\rho)$ of $\lf$ on the vector space $M$ defines an $\lf$-bimodule
structure on $M$ via $x\cdot m:=\lambda_x(m)$ and $m\cdot x:=\rho_x(m)$.

By virtue of \cite[Lemma 3.3]{F}, every left $\lf$-module is an $\lf_\lie$-module in a
natural way, and vice versa. Consequently, many properties of left Leibniz modules
follow from the corresponding properties of modules for the canonical Lie algebra.

The usual definitions of the notions of {\em sub(bi)module\/}, {\em irreducibility\/},
{\em complete reducibility\/}, {\em composition series\/}, {\em homomorphism\/},
{\em isomorphism\/}, etc., hold for left Leibniz modules and Leibniz bimodules. (Note
that by definition an irreducible Leibniz (bi)module is always non-zero.)

Let $\lf$ be a left Leibniz algebra over a field $\F$, and let $M$ be an $\lf$-bimodule.
Then $M$ is said  to be {\em symmetric\/} if $m\cdot x=-x\cdot m$ for every $x\in
\lf$ and every $m\in M$, and $M$ is said  to be {\em anti-symmetric\/} if $m\cdot
x=0$ for every $x\in\lf$ and every $m\in M$. Moreover, an $\lf$-bimodule $M$ is
called {\em trivial\/} if $x\cdot m=0=m\cdot x$ for every $x\in\lf$ and every $m\in
M$. Note that an $\lf$-bimodule $M$ is trivial if, and only if, $M$ is symmetric and
anti-symmetric. We call
$$M_0:=\langle x\cdot m+m\cdot x\mid x\in\lf,m\in M\rangle_\F$$
the {\em anti-symmetric kernel\/} of $M$. It is well known that $M_0$ is an
anti-symmetric $\lf$-subbimodule of $M$ such that $M_\sym:=M/M_0$ is symmetric
(see \cite[Proposition~3.12 and Proposition 3.13]{F}).

Recall that every left $\lf$-module $M$ of a left Leibniz algebra $\lf$ determines
a unique symmetric $\lf$-bimodule structure on $M$ by defining $m\cdot x:=
-x\cdot m$ for every element $m\in M$ and every element $x\in\lf$ (see
\cite[Proposition 3.15\,(b)]{F}). We will denote this symmetric $\lf$-bimodule
by $M_s$. Similarly, every left $\lf$-module $M$ with trivial right action is an
anti-symmetric $\lf$-bimodule (see \cite[Proposition 3.15\,(a)]{F}) which will
be denoted by $M_a$.

Recall that for a subset $S$ of a left Leibniz algebra $\lf$ the {\em space of
right $S$-invariants\/} of an $\lf$-bimodule $M$ is
$$M^S:=\{m\in M\mid\forall\,s\in S:m\cdot s=0\}\,.$$
In particular, we have that $M^\emptyset:=M$.

Our first result is an obvious generalization of \cite[Lemma 1.2]{FW}.

\begin{lem}\label{inv}
Let $\lf$ be a left Leibniz algebra, let $\iif$ be a left ideal of $\lf$, and let $M$ be an
$\lf$-bimodule. Then $M^\iif$ is an $\lf$-subbimodule of $M$.
\end{lem}

\begin{proof}
It follows from (LML) that $M^\iif$ is invariant under the left $\lf$-action on $M$, and it
follows from (MLL) that $M^\iif$ is invariant under the right $\lf$-action on $M$.
\end{proof}

Let $\lf$ be a left Leibniz algebra over a field $\F$, and let $M$ be an $\lf$-bimodule
with associated representation $(\lambda,\rho)$. We say that $\ann_\lf^\ell(M):=
\Ker(\lambda)$ is the {\em left annihilator\/} of $M$. Similarly, $\ann_\lf^r(M):=
\Ker(\rho)$ is the {\em right annihilator\/} of $M$, and $\ann_\lf(M):=\ann_\lf^\ell
(M)\cap\ann_\lf^r(M)$ is called the {\em annihilator\/} of $M$.

It is clear from the definition of $M^\lf$ that an $\lf$-bimodule $M$ is anti-symmetric
if, and only if, $M^\lf=M$. The following generalization of \cite[Lemma 1.1]{FW}
will be useful later in this paper:

\begin{lem}\label{sym}
Let $\lf$ be a left Leibniz algebra, let $S$ be a subset of $\lf$, and let $M$ be an
$\lf$-bimodule such that $M^S=0$. Then $M$ is symmetric. In particular, $\leib
(\lf)\subseteq\ann_\lf(M)$.
\end{lem}

\begin{proof}
Since $M_0$ is anti-symmetric, it follows from the hypothesis that $$M_0=
M_0^S\subseteq M^S=0\,.$$ Hence we obtain from the definition of $M_0$
that $M$ is symmetric. The second part is then an immediate consequence
of \cite[Lemma 3.10]{F}.
\end{proof}

We continue with two results that will be needed at the end  of this section.

\begin{lem}\label{triv}
Let $\lf$ be a left Leibniz algebra, and let $M$ be an $\lf$-bimodule. Then
$M^\lf=0$ if, and only if, $M$ is symmetric, and $M$ does not contain a
non-zero trivial $\lf$-subbimodule.
\end{lem}

\begin{proof}
One implication is an immediate consequence of Lemma \ref{sym}. Conversely,
suppose that $M^\lf\ne 0$. If $M$ is not symmetric, then there is nothing to
prove. On the other hand, if $M$ is symmetric, then it follows from Lemma
\ref{inv} that $M^\lf$ is a non-zero trivial $\lf$-subbimodule of $M$. 
\end{proof}

Note that Lemma \ref{triv} generalizes \cite[Corollary 1.3]{FW} from irreducible
to arbitrary bimodules. Moreover, for $S=\lf$ one implication of Lemma \ref{triv}
is the converse of Lemma \ref{sym}.

Lemma \ref{1dim} generalizes Lemma \ref{triv} further, but needs a
stronger hypothesis on the ground field and seems to hold only for
finite-dimensional bimodules. Recall that $\lf\lf:=\langle xy\mid x,y\in
\lf\rangle_\F$ is the {\em derived subalgebra\/} of a left Leibniz
algebra $\lf$ (see \cite{S}).

\begin{lem}\label{1dim}
Let $\lf$ be a left Leibniz algebra over an algebraically closed field $\F$, and
let $M$ be a finite-dimensional $\lf$-bimodule. Then $M^{\lf\lf}=0$ if, and
only if, $M$ is symmetric, and $M$ does not contain a one-dimensional
$\lf$-subbimodule.
\end{lem}

\begin{proof}
Assume first that $M^{\lf\lf}=0$. By virtue of Lemma \ref{sym}, we have
that $M$ is symmetric. Suppose now that $M$ contains a one-dimensional
$\lf$-subbimodule $\F m_0$. Then there exists a linear form $\mu\in\lf^*$
on $\lf$ such that $x\cdot m_0=\mu(x)m_0$ for every element $x\in\lf$,
and thus we obtain from (LLM) that
$$(xy)\cdot m_0=x\cdot(y\cdot m_0)-y\cdot(x\cdot m_0)=\mu(x)\mu(y)m_0
-\mu(x)\mu(y)m_0=0$$
for any elements $x,y\in\lf$. Hence $0\ne m_0\in M^{\lf\lf}$ which is a
contradiction.

On the other hand, assume that $M$ is symmetric and does not contain
a one-dimensional $\lf$-subbimodule. Furthemore, suppose that $N:=
M^{\lf\lf}\ne 0$. Then $N$ is a finite-dimensional $(\lf/\lf\lf)$-bimodule.
Since $\dim_\F N<\infty$, we conclude that $N$ contains an irreducible
left $(\lf/\lf\lf)$-submodule $U$. As $\lf/\lf\lf$ is an abelian Lie algebra
and the ground field $\F$ is algebraically closed, we obtain that $\dim_\F
U=1$. But then $U_s$ is also a one-dimensional $\lf$-subbimodule
of $M$ which contradicts the hypothesis.
\end{proof}

The last two results of this section will be important in Section \ref{solv}.

\begin{pro}\label{nontriv}
Let $\lf$ be a left Leibniz algebra, and let $M$ be a finite-dimensional
$\lf$-bimodule such that every composition factor of $M$ is non-trivial.
Then $M^\lf=M_0$, and if in addition $M$ is symmetric, then $M^\lf=0$.
\end{pro}

\begin{proof}
Firstly, for symmetric $\lf$-bimodules the assertion is an immediate
consequence of Lemma \ref{triv}.

Next, if $M$ is arbitrary, then in the short exact sequence
$$0\to M_0\to M\to M_\sym\to 0$$
the first term is anti-symmetric and the third term is symmetric
(see \cite[Propositions~3.12 and 3.13]{F}). Consequently, an
application of the long exact cohomology sequence in conjunction
with the statement for the symmetric case yields the exact sequence
$$0\to M_0=M_0^\lf\hookrightarrow M^\lf\to M_\sym^\lf=0\,,$$
which then implies that $M^\lf=M_0$.
\end{proof}

The next example shows that the converse of Proposition \ref{nontriv}
is not true.
\vspace{.3cm}

\noindent {\bf Example A.} Let $\lf:=\F e$ be the one-dimensional Lie
algebra. Consider the matrices
$$A:=\left[\begin{array}{ccc}0 & 1 & 0\\0 & 0 & 1\\0 & 0 & 0\end{array}\right]$$
and
$$B:=\left[\begin{array}{ccc}0 & 0 & 1\\0 & 0 & 0\\0 & 0 & 0\end{array}\right]\,.$$
Then the vector space $M:=\F^3$ is a Leibniz $\lf$-bimodule via $\lambda_e(m):=
Am$ and $\rho_e(m):=Bm$ for any column vector $m\in M$ because the matrices
$A$ and $B$ satisfy the identities $AB=BA$ and $B^2=-BA$.
Note that we have
$$M^\lf=\Ker(\rho_e)=\langle\left[\begin{array}{c}1\\0\\0\end{array}\right],
\left[\begin{array}{c}0\\1\\0\end{array}\right]\rangle_\F$$
and
$$M_0=\im(\lambda_e+\rho_e)=\langle\left[\begin{array}{c}1\\0\\0\end{array}
\right],\left[\begin{array}{c}1\\1\\0\end{array}\right]\rangle_\F\,,$$
which imply that $M^\lf=M_0$, but $M$ contains the trivial $\lf$-subbimodule
$$\Ker(\lambda_e)\cap\Ker(\rho_e)=\langle\left[\begin{array}{c}1\\0\\0
\end{array}\right]\rangle_\F\,.$$
\vspace{.3cm}

Finally, a proof very similar to the one for Proposition \ref{nontriv}, in which
one uses Lemma \ref{1dim} instead of Lemma \ref{triv}, shows:

\begin{pro}\label{non1dim}
Let $\lf$ be a left Leibniz algebra over an algebraically closed field $\F$,
and let $M$ be a finite-dimensional $\lf$-bimodule such that no composition
factor of $M$ is one-dimensional. Then $M^{\lf\lf}=M_0$, and if in addition
$M$ is symmetric, then $M^{\lf\lf}=0$.
\end{pro}


\section{Fitting decomposition of a Leibniz bimodule}\label{fitting}


Let $V$ be a vector space over a field $\F$, and let $T\in\End_\F(V)$ be
a linear operator on $V$. Then $T$ is called {\em locally nilpotent\/} if
for every vector $v\in V$ there exists a positive integer $n(v)$ such that
$T^{n(v)}(v)=0$.

Moreover, we recall {\em Fitting's lemma for linear operators\/} which
asserts that for every linear operator $T$ on a finite-dimensional vector
space $V$ one has
$$V=V_0(T)\oplus V_1(T)$$
such that $T_{\vert V_0(T)}$ is nilpotent and $T_{\vert V_1(T)}$ is
invertible, where
$$V_0(T):=\bigcup_{n\in\N}\Ker(T^n)$$
and
$$V_1(T):=\bigcap_{n\in\N}\im(T^n)$$
are $T$-invariant subspaces of $V$ (see \cite[Section 4 of Chapter II]{J}).
In \cite[Theorem~4 of Chapter II]{J} Fitting's lemma is generalized to
nilpotent Lie algebras of linear operators. For our purposes we will need
an even slightly more general version of Fitting's lemma which we will
prove at the end of this section. In order to be able do so, we show
the following two results (see \cite[Lemma 1 of Chapter II]{J} and
\cite[Lemma~4.6]{F2} for Lie algebra versions of Lemma \ref{fitting0}).
The only new insight is that the two identities (\ref{LLMrep}) and
(\ref{LMLrep}) are sufficient to carry out the proofs.

\begin{lem}\label{fitting0}
Let $\lf$ be a left Leibniz algebra, and let $M$ be a left $\lf$-module with
associated left representation $\lambda:\lf\to\gl(M)$. If $a\in\lf$ is an
element such that the left multiplication operator $L_a:\lf\to\lf$, $x\mapsto
ax$ is locally nilpotent, then $M_0(\lambda_a)$ is an $\lf$-subbimodule
of $M$.
\end{lem}

\begin{proof}
Let $x$ and $y$ be arbitrary elements of $\lf$. Using (\ref{LLMrep}) and
(\ref{LMLrep}) one can prove by induction on $n$ that
\begin{equation}\label{ll}
\lambda_x^n\circ\lambda_y=\sum_{k=0}^n{n\choose k}\lambda_{L_x^k(y)}
\circ\lambda_x^{n-k}
\end{equation}
and
\begin{equation}\label{lr}
\lambda_x^n\circ\rho_y=\sum_{k=0}^n{n\choose k}\rho_{L_x^k(y)}\circ
\lambda_x^{n-k}
\end{equation}
for every positive integer $n$.

Now let $m\in M_0(\lambda_a)$ and $x\in\lf$ be arbitrary. Since $m\in M_0
(\lambda_a)$ and $L_a$ is locally nilpotent, there exists a positive integer $n$
such that $\lambda_a^n(m)=0$ and $L_a^n(x)=0$.

Firstly, we obtain from (\ref{ll}) that
\begin{eqnarray*}
\lambda_a^{2n}(x\cdot m) & = & (\lambda_a^{2n}\circ\lambda_x)(m)\\
& = & \sum_{k=0}^{2n}{2n\choose k}\left(\lambda_{L_a^k(x)}\circ\lambda_a^{2n-k}
\right)(m)\\
& = & \sum_{k=0}^{n-1}{2n\choose k}\left(\lambda_{L_a^k(x)}\circ\lambda_a^{2n-k}
\right)(m)\\
& = & 0\,,
\end{eqnarray*}
which shows that $x\cdot m\in M_0(\lambda_a)$.

Next, we obtain from (\ref{lr}) that
\begin{eqnarray*}
\lambda_a^{2n}(m\cdot x) & = & (\lambda_a^{2n}\circ\rho_x)(m)\\
& = & \sum_{k=0}^{2n}{2n\choose k}\left(\rho_{L_a^k(x)}\circ\lambda_a^{2n-k}
\right)(m)\\
& = & \sum_{k=0}^{n-1}{2n\choose k}\left(\rho_{L_a^k(x)}\circ\lambda_a^{2n-k}
\right)(m)\\
& = & 0\,,
\end{eqnarray*}
which shows that $m\cdot x\in M_0(\lambda_a)$.
\end{proof}

It should be mentioned that we will only need Lemma \ref{fitting0} in the
remainder of the paper. But Theorem \ref{fitting} below might be useful
when one studies other aspects of Leibniz bimodules.

\begin{lem}\label{fitting1}
Let $\lf$ be a left Leibniz algebra, and let $M$ be a finite-dimensional left
$\lf$-module with associated left representation $\lambda:\lf\to\gl(M)$. If
$a\in\lf$ is an element such that the left multiplication operator $L_a:\lf
\to\lf$, $x\mapsto ax$ is locally nilpotent, then $M_1(\lambda_a)$ is an
$\lf$-subbimodule of $M$.
\end{lem}

\begin{proof}
Let $x$ and $y$ be arbitrary elements of $\lf$. Using (\ref{LLMrep}) and
(\ref{LMLrep}) one can prove by induction on $n$ that
\begin{equation}\label{rl}
\lambda_y\circ\lambda_x^n=\sum_{k=0}^n(-1)^{n-k}{n\choose k}
\lambda_x^k\circ\lambda_{L_x^{n-k}(y)}
\end{equation}
and
\begin{equation}\label{rr}
\rho_y\circ\lambda_x^n=\sum_{k=0}^n(-1)^{n-k}{n\choose k}
\lambda_x^k\circ\rho_{L_x^{n-k}(y)}
\end{equation}
for every positive integer $n$.

Let $r$ be the smallest positive integer such that
$$\im(\lambda_a^r)=\im(\lambda_a^{r+1})=\cdots=M_1(\lambda_a)$$
and let $s$ be the smallest positive integer such that
$$\Ker(\lambda_a^s)=\Ker(\lambda_a^{s+1})=\cdots=M_0(\lambda_a)\,,$$
and set $t:=\max\{r,s\}$ (see \cite[Section 4 of Chapter II]{J}). 

Now let $m\in M_1(\lambda_a)$ and $x\in\lf$ be arbitrary. Since $L_a$ is
locally nilpotent, there exists a positive integer $n$ such that $L_a^n(x)=0$.
Moreover, there exists an element $m_0\in M$ such that $m=\lambda_a^{
t+n-1}(m_0)$. Note that the hypothesis $\dim_\F M<\infty$ implies that the
integers $r$ and $s$ always exist.

Firstly, we obtain from (\ref{rl}) that
\begin{eqnarray*}
x\cdot m & = & x\cdot\lambda_a^{t+n-1}(m_0)\\
& = & (\lambda_x\circ\lambda_a^{t+n-1})(m_0)\\
& = & \sum_{k=0}^{t+n-1}(-1)^{t+n-1-k}{t+n-1\choose k}\left(\lambda_a^k
\circ\lambda_{L_a^{t+n-1-k}(x)}\right)(m_0)\\
& = & \sum_{k=t}^{t+n-1}(-1)^{t+n-1-k}{t+n-1\choose k}\lambda_a^k
\left(\lambda_{L_a^{t+n-1-k}(x)}(m_0)\right)\in\im(\lambda_a^t)=M_1
(\lambda_a)\,.
\end{eqnarray*}

Next, we obtain from (\ref{rr}) that
\begin{eqnarray*}
m\cdot x & = & \lambda_a^{t+n-1}(m_0)\cdot x\\
& = & (\rho_x\circ\lambda_a^{t+n-1})(m_0)\\
& = & \sum_{k=0}^{t+n-1}(-1)^{t+n-1-k}{t+n-1\choose k}\left(\lambda_a^k
\circ\rho_{L_a^{t+n-1-k}(x)}\right)(m_0)\\
& = & \sum_{k=t}^{t+n-1}(-1)^{t+n-1-k}{t+n-1\choose k}\lambda_a^k
\left(\rho_{L_a^{t+n-1-k}(x)}(m_0)\right)\in\im(\lambda_a^t)=M_1
(\lambda_a)\,,
\end{eqnarray*}
which completes the proof.
\end{proof}

By abuse of language we write
$$M_0(S):=\bigcap_{s\in S}M_0(\lambda_s)$$
and
$$M_1(S):=\sum_{s\in S}M_1(\lambda_s)\,.$$

The main result of this section is the following {\em Fitting lemma for Leibniz
algebras\/}. Note that our proof of part (d) will reduce the statement
to the corresponding statement of \cite[Theorem 4 of Chapter II]{J}.

\begin{thm}\label{fitting}
Let $\lf$ be a left Leibniz algebra over a field $\F$, and let $M$ be a $\lf$-bimodule
with associated representation $(\lambda,\rho)$. If $S$ is a subset of $\lf$ such
that the left multiplication operator $L_s:\lf\to\lf$, $x\mapsto sx$ is locally nilpotent
for every element $s\in S$, then the following statements hold:
\begin{enumerate}
\item[{\rm (a)}] $M_0(S)$ is an $\lf$-subbimodule of $M$.
\item[{\rm (b)}] Every element of $S$ acts locally nilpotently on
                         $M_0(S)$ from the left and from the right.
\end{enumerate}
Moreover, if $\dim_\F M<\infty$, then
\begin{enumerate}
\item[{\rm (c)}] $M_1(S)$ is an $\lf$-subbimodule of $M$.
\item[{\rm (d)}] $M=M_0(S)\oplus M_1(S)$.
\end{enumerate}
\end{thm}

\begin{proof}
(a) is an immediate consequence of the definition of $M_0(S)$ and Lemma
\ref{fitting0}.

(b): For the left action of $S$ the assertion is an immediate consequence of
the definition of $M_0(S)$ and for the right action the claim then follows from
\cite[Lemma~6]{P1}.

(c): Since a sum of $\lf$-subbimodules is again an $\lf$-subbimodule, the
assertion is an immediate consequence of the definition of $M_1(S)$ and
Lemma \ref{fitting1}.

(d): It follows from \cite[Lemma 3.3]{F} that $M$ is a left $\lf_\lie$-module
via $\overline{x}\cdot m:=x\cdot m$ for any elements $x\in\lf$ and $m\in M$.
Let $\overline{\lambda}:\lf_\lie\to\gl(M)$ denote the corresponding representation
of $\lf_\lie$ on $M$. Then we have that $\overline{\lambda}_{\overline{x}}=
\lambda_x$ for every element $x\in\lf$. In particular, we obtain that
$$M_0(S)=\bigcap_{s\in S}M_0(\lambda_s)=\bigcap_{s\in S}M_0
(\overline{\lambda}_{\overline{s}})$$
and
$$M_1(S)=\sum_{s\in S}M_1(\lambda_s)=\sum_{s\in S}M_1
(\overline{\lambda}_{\overline{s}})\,,$$
and therefore the assertion follows from \cite[Theorem 4 in Chapter II]{J} applied
to the Lie subalgebra $\gf:=\langle\overline{\lambda}_{\overline{s}}\mid s\in S
\rangle_\F$ of $\gl(M)$. Namely, it follows from identity (\ref{LLMrep}) by induction
on $n$ that
$$\left(\ad_\gf\overline{\lambda}_{\overline{s}}\right)^n\left(
\overline{\lambda}_{\overline{t}}\right)=\lambda_{L_s^n(t)}$$
for any elements $s,t\in S$. Then this in conjunction with Engel's theorem (see
\cite[Theorem 3.2]{H}) shows that $\gf$ is nilpotent, and thus Jacobson's result
can be applied.
\end{proof}

The following example shows that in Theorem \ref{fitting} the elements of $S$
do not necessarily act invertibly on the Fitting-1-component $M_1(S)$ of a left
Leibniz module or a Leibniz bimodule $M$. 
\vspace{.3cm}

\noindent {\bf Example B.} Let $\gf$ be the abelian Lie subalgebra of the general
linear Lie algebra $\gl_2(\F)$ generated by
$$E=\left[\begin{array}{cc}0 & 1\\0 & 0\end{array}\right]$$
and
$$I=\left[\begin{array}{cc}1 & 0\\0 & 1\end{array}\right]\,.$$
Then $M:=\F^2$ is a left $\gf$-module via $\lambda_X(m):=Xm$ for any matrix
$X\in\gf$ and any column vector $m\in M$. If we set $S:=\{E,I\}$, then $M_0(S)
=\{0\}$ and $M_1(S)=M$, but clearly $E$ does not act invertibly
on $M$. (Note that $M$ can be made into a Leibniz $\gf$-bimodule
by considering its symmetrization $M_s$ or its anti-symmetrization $M_a$.)


\section{Cohomology of non-semi-simple Leibniz algebras}\label{hh}


Similar to the coboundary operator in \cite[Section 1.8]{LP} for the cohomology
of a right Leibniz algebra with coefficients in a Leibniz bimodule one can also
introduce a coboundary operator $\dl^\bullet$ for the cohomology of a left Leibniz
algebra with coefficients in a Leibniz bimodule as follows. Let $\lf$ be a left
Leibniz algebra over a field $\F$, and let $M$ be an $\lf$-bimodule. For every
non-negative integer $n$ set $\CL^n(\lf,M):=\Hom_\F(\lf^{\otimes n},M)$
and consider the linear transformation $\dl^n:\CL^n(\lf,M)\to\CL^{n+1}(\lf,M)$
defined by
\begin{eqnarray*}
(\dl^nf)(x_1\otimes\dots\otimes x_{n+1}) & := & \sum_{j=1}^n(-1)^{j+1}
x_j\cdot f(x_1\otimes\dots\otimes\hat{x}_j\otimes\dots\otimes x_{n+1})\\
& + & (-1)^{n+1}f(x_1\otimes\dots\otimes x_n)\cdot x_{n+1}\\
& + & \sum_{1\le i<j\le n+1}(-1)^if(x_1\otimes\dots\otimes\hat{x}_i\otimes
\dots\otimes x_{j-1}\otimes x_ix_j\otimes x_{j+1}\otimes\dots\otimes x_{n+1})
\end{eqnarray*}
for any $f\in\CL^n(\lf,M)$ and any elements $x_1,\dots,x_{n+1}\in\lf$.

It is proved in \cite[Lemma 1.3.1]{C} that $\CL^\bullet(\lf,M):=(\CL^n(\lf,M),
\dl^n)_{n\in\N_0}$is a cochain complex, i.e., $\dl^{n+1}\circ\dl^n=0$ for
every non-negative integer $n$. Hence, one can define the {\em cohomology
of $\lf$ with coefficients in an $\lf$-bimodule\/} $M$ by
$$\HL^n(\lf,M):=\HCE^n(\CL^\bullet(\lf,M)):=\ZL^n(\lf,M)/\BL^n(\lf,M)$$
for every non-negative integer $n$, where
$$\ZL^n(\lf,M):=\Ker(\dl^n)\mbox{ and }\BL^n(\lf,M):=\im(\dl^{n-1})\,.$$
(Note that $\dl^{-1}:=0$.)

Moreover, we will also need the linear operator $\theta_a^n:\CL^n(\lf,M)\to
\CL^n(\lf,M)$ defined by
\begin{eqnarray*}
\theta_a^n(f)(x_1\otimes\dots\otimes x_n) & := & a\cdot f(x_1\otimes
\dots\otimes x_n)-\sum_{j=1}^nf(x_1\otimes\dots\otimes x_{j-1}\otimes
ax_j\otimes x_{j+1}\otimes\dots\otimes x_n)
\end{eqnarray*}
for any $f\in\CL^n(\lf,M)$ and any elements $a,x_1,\dots,x_n\in\lf$ as well
as the linear transformation $\iota_a^n:\CL^n(\lf,M)\to\CL^{n-1}(\lf,M)$
defined by
$$\iota_a^n(f)(x_1\otimes\dots\otimes x_{n-1}):=f(a\otimes x_1\otimes\dots
\otimes x_{n-1})$$
for any $f\in\CL^n(\lf,M)$ and any elements $a,x_1,\dots,x_{n-1}\in\lf$.

Then the following identities hold for every element $a\in\lf$ (see
\cite[Proposition 1.3.2\,(1) \& (4)]{C}):
\begin{equation}\label{cartan}
\dl^{n-1}\circ\iota_a^n+\iota_a^{n+1}\circ\dl^n=\theta_a^n
\end{equation}
for every positive integer $n$, and
\begin{equation}\label{invact}
\theta_a^{n+1}\circ\dl^n=\dl^n\circ\theta_a^n
\end{equation}
for every non-negative integer $n$.

\newpage

Our first result in this section is the Leibniz analogue of \cite[Corollary 4.3]{F2}.

\begin{lem}\label{invert}
Let $V$ and $W$ be left modules over a left Leibniz algebra $\lf$. If $x$ is an
element of $\lf$ such that
\begin{enumerate}
\item[{\rm (i)}]  $x$ acts locally nilpotently on $V$, and
\item[{\rm (ii)}] $x$ acts invertibly on $W$,
\end{enumerate}
then $x$ acts invertibly on $\Hom_\F(V,W)$.
\end{lem}

\begin{proof}
Similarly to the proof in \cite[Lemma 1.4\,(b)]{FW}, one can show that $\Hom_\F
(V,W)$ is a left $\lf$-module via
$$(x\cdot f)(y) := x\cdot f(y)-f(xy)$$
for every $f\in\Hom_\F(V,W)$ and any elements $x,y\in\lf$. Then the assertion is
a special case of \cite[Lemma 4.2]{F2}.
\end{proof}

The next two results are the Leibniz analogues of results that Farnsteiner
obtained for Hochschild cohomology (see \cite[Theorem 4.4]{F2} and
\cite[Theorem 4.7]{F2}). The proofs follow those in \cite{F2} very closely,
but for the convenience of the reader we include the details.

\begin{thm}\label{vanhh}
Let $\lf$ be a left Leibniz algebra, and let $M$ be an $\lf$-bimodule with
associated representation $(\lambda,\rho)$. If $a$ is an element of $\lf$
such that
\begin{enumerate}
\item[{\rm (i)}]  $L_a:\lf\to\lf$, $x\mapsto ax$ is locally nilpotent, and
\item[{\rm (ii)}] $\lambda_a:M\to M$ is invertible,
\end{enumerate}
then $\HL^n(\lf,M)=0$ for every positive integer $n$. Moreover, if $M$
is symmetric, then $\HL^n(\lf,M)=0$ for every non-negative integer $n$.
\end{thm}

\begin{proof}
Let $n$ be an arbitrary non-negative integer. Then it follows similarly as in
the proof of \cite[Proposition 1.3.2\,(2)]{C} that $\lf^{\otimes n}$ is a left
$\lf$-module via
$$\tau_a^n(x_1\otimes\dots\otimes x_n):=\sum_{j=1}^n x_1\otimes
\dots\otimes x_{j-1}\otimes ax_j\otimes x_{j+1}\otimes\dots\otimes x_n\,.$$
Since by hypothesis $L_a$ is locally nilpotent on $\lf$, we conclude that
$\tau_a^n$ is locally nilpotent on $\lf^{\otimes n}$.

Recall that $\theta^n$ denotes the representation of $\lf$ on $\CL^n(\lf,M)
:=\Hom_\F(\lf^{\otimes n},M)$ obtained from $\tau^n$ and $\lambda$.
It follows from Lemma \ref{invert} that $\theta_a^n$ is invertible. By
virtue of identity (\ref{invact}), $\theta^n$ induces a representation 
$\overline{\theta}^n$ of $\lf$ on $\HL^n(\lf,M)$ via
$$\overline{\theta}^n_x(f+\BL^n(\lf,M)):=\theta_x^n(f)+\BL^n(\lf,M)\,.$$
As a consequence, we deduce that $\overline{\theta}_a^n$ is invertible
on $\HL^n(\lf,M)$.

Let $n$ be any positive integer. Then we obtain from the identity (\ref{cartan})
that $\theta_x^n(f)=\dl^{n-1}(\iota_x^n(f))\in\BL^n(\lf,M)$ for every $x\in\lf$
and every $f\in\ZL^n(\lf,M)$ which implies that $\overline{\theta}^n$ is the
trivial representation of $\lf$ on $\HL^n(\lf,M)$. Consequently, we have that
$\overline{\theta}_a^n=0$ is an invertible linear operator on $\HL^n(\lf,M)$,
and therefore $\HL^n(\lf,M)=0$.

On the other hand, if $M$ is symmetric, then $a$ also acts invertibly and trivially
on $M^\lf$ which yields in addition that $\HL^0(\lf,M)=M^\lf=0$.
\end{proof}

\begin{thm}\label{fittinghh}
Let $\lf$ be a left Leibniz algebra over a field $\F$, and let $M$ be an $\lf$-bimodule.
If $S$ is a subset of $\lf$ such that
\begin{enumerate}
\item[{\rm (i)}]  $L_s:\lf\to\lf$, $x\mapsto sx$ is locally nilpotent for every element
                         $s\in S$, and
\item[{\rm (ii)}] $\dim_\F M/M_0(S)<\infty$,
\end{enumerate}
then $\HL^n(\lf,M)\cong\HL^n(\lf,M_0(S))$ {\rm (}\hspace{-.5mm}as $\F$-vector
spaces{\rm )} for every integer $n\ge 2$. Moreover, if $M$ is symmetric, then
$\HL^n(\lf,M)\cong\HL^n(\lf,M_0(S))$ for every non-negative integer $n$.
\end{thm}

\begin{proof}
We proceed by induction on $d:=\dim_\F M/M_0(S)$. The base step $d=0$ is
trivially true. So let $d>0$. Then there exists an element $s_0\in S$ such that
$N:=M_0(\lambda_{s_0})\ne M$. Note that it follows from Lemma \ref{fitting0}
that $N$ is an $\lf$-subbimodule of $M$.

Next, we show that $N_0(S)=M_0(S)$. As $N\subseteq M$, we have that
$N_0(S)\subseteq M_0(S)$. In order to prove the reverse inclusion, let
$m\in M_0(S)$ be an arbitrary element. As $s_0\in S$, we have that
$m\in M_0(S)\subseteq M_0(\lambda_{s_0})=N$, and thus $m\in N_0(S)$.

Because of $N_0(S)=M_0(S)$, we have $\dim_\F N/N_0(S)<d$, and therefore
the induction hypothesis yields that
$$\HL^n(\lf,N)\cong\HL^n(\lf,N_0(S))=\HL^n(\lf,M_0(S))$$
for any integer $n\ge 2$. Now from the short exact sequence
$$0\to N\to M\to M/N\to 0$$
of $\lf$-bimodules we obtain the long exact cohomology sequence
$$\cdots\to\HL^{n-1}(\lf,M/N)\to\HL^n(\lf,N)\to\HL^n(\lf,M)\to\HL^n(\lf,M/N)
\to\cdots\,.$$
It follows from the definition of $N$ that the linear operator on the finite-dimensional
vector space $M/N$ induced by $\lambda_{s_0}$ is injective, and therefore invertible.
As a consequence, we deduce from Theorem \ref{vanhh} that $\HL^n(\lf,M/N)=0$
for every integer $n\ge 1$, and thus we conclude
$$\HL^n(\lf,M)\cong\HL^n(\lf,N)\cong\HL^n(\lf,M_0(S))$$
for every integer $n\ge 2$.

On the other hand, if $M$ is symmetric, then we obtain the same conclusion
for every integer $n\ge 0$.
\end{proof}

\noindent {\bf Remark 1.} The example in \cite[Section 6, p.\ 663]{F2} in
conjunction with \cite[Theorem~2.6]{FW} shows that hypothesis (ii) in
Theorem \ref{fittinghh} is necessary. Moreover, if $\dim_\F M<\infty$,
then it follows from the Fitting decomposition for the linear operator
$\lambda_{s_0}$ in conjunction with Theorem \ref{vanhh} that Theorem
\ref{fittinghh} remains true for $n=1$. We believe that Theorem \ref{fittinghh}
is always true for $n=1$, but at the moment we do not know how to prove this.
\vspace{.3cm}

Recall that an algebra $\rf$ is called {\em nilpotent\/} if there exists a positive
integer $n$ such that any product of $n$ elements in $\rf$, no matter how
associated, is zero (see \cite{S}). Let $\lf$ be a left Leibniz algebra. Then the
{\em left descending central series\/}
$$^1\hspace{-.5mm}
\lf\supseteq\hspace{.1mm}^2\hspace{-.5mm}\lf\supseteq\hspace{.1mm}
^3\hspace{-.5mm}\lf\supseteq\cdots$$
of $\lf$ is defined recursively by $^1\hspace{-.5mm}\lf:=\lf$ and
$^{r+1}\hspace{-.5mm}\lf:=\lf\hspace{.5mm}(^r\hspace{-.5mm}\lf)$ for
every positive integer $r$. It follows from \cite[Proposition 5.3 and Lemma
5.6]{F} that a left Lie algebra $\lf$ is nilpotent exactly when there exists a
positive integer $m$ such that $^m\hspace{-.5mm}\lf=0$.

The following immediate consequence of Theorem \ref{fittinghh} is the analogue
of \cite[Corollary~6.3]{F2} for Leibniz cohomology:

\begin{cor}\label{cohfitting}
Let $\lf$ be a left Leibniz algebra over a field $\F$, and let $\nf$ be a nilpotent
right ideal of $\lf$. If $M$ is an $\lf$-bimodule such that $\dim_\F M/M_0(\nf)
<\infty$, then $\HL^n(\lf,M)\cong\HL^n(\lf,M_0(\nf))$ {\rm (}\hspace{-.5mm}as
$\F$-vector spaces{\rm )} for every integer $n\ge 2$. Moreover, if $M$ is symmetric,
then $\HL^n(\lf,M)\cong\HL^n(\lf,M_0(\nf))$ for every non-negative integer $n$.
\end{cor}

\begin{proof}
According to Theorem \ref{fittinghh}, we need only to show that $L_a:\lf\to\lf$
is locally nilpotent for every element $a\in\nf$. Since $\nf$ is nilpotent, this
follows from $L_a^r(x)\in\hspace{.1mm}^r\hspace{-.1mm}\nf$ for any
element $x\in\lf$ and any positive integer $r$.
\end{proof}

\noindent {\bf Remark 2.} The proof of Corollary \ref{cohfitting} shows that
$L_a$ is nilpotent for every element $a\in\nf$. In fact, the same exponent
can be chosen for any such element. According to Remark 1, Corollary
\ref{cohfitting} also holds for $n=1$ provided $\dim_\F M<\infty$.
\vspace{.3cm}

As a consequence of the previous result we obtain the following vanishing theorem
for Leibniz cohomology which will be useful in the next section.

\begin{cor}\label{van}
Let $\lf$ be a left Leibniz algebra, and let $\nf$ be a nilpotent right ideal of $\lf$.
If $M$ is a finite-dimensional $\lf$-bimodule such that $M^\nf=0$, then $\HL^n
(\lf,M)=0$ for every non-negative integer $n$.
\end{cor}

\begin{proof}
It follows from Lemma \ref{sym} in conjunction with the hypothesis $M^\nf=0$
that $M$ is symmetric. Suppose that $M_0(\nf)\ne 0$. Since by definition the
elements of $\nf$ act nilpotently from the left on the $\lf$-subbimodule $M_0
(\nf)$ of $M$, the Leibniz analogue of Engel’s theorem for Lie algebras of linear
transformations (see \cite[Theorem 7]{P1} or \cite[Theorem 5.17]{F}) implies
that $M_0(\nf)^\nf\ne 0$. Consequently, we have that $0\ne M_0(\nf)^\nf
\subseteq M^\nf$ which contradicts the hypothesis. We conclude that $M_0
(\nf)=0$, and thus Corollary \ref{cohfitting} yields the assertion.
\end{proof}

As a first application of the previous results we conclude this section by proving
the following analogues of \cite[Theorem~2.4]{F1} for Leibniz cohomology.

We say that an $\lf$-bimodule $M$ with associated representation $(\lambda,
\rho)$ is {\em right faithful\/} if its right annihilator $\ann_\lf^r(M):=\Ker
(\rho)$ is zero.

\begin{cor}\label{cohnonsemisim}
Let $\lf$ be a left Leibniz algebra, and let $M$ be a finite-dimensional right
faithful irreducible $\lf$-bimodule. If $\lf$ is either a non-semi-simple Lie
algebra or any non-Lie Leibniz algebra, then $\HL^n(\lf,M)=0$ for every
non-negative integer $n$.
\end{cor}

\begin{proof}
Either $\lf$ is a Lie algebra or $\leib(\lf)\ne 0$. In the first case there exists
a non-zero abelian ideal of $\lf$ and in the other case $\leib(\lf)$ is such an
ideal which we both call $\af$. If $M^\af=0$, the assertion follows from
Corollary \ref{van}. Otherwise we obtain from Lemma \ref{inv} that $M^\af$
is a non-zero $\lf$-subbimodule of $M$. As $M$ is irreducible, this implies
that $M^\af=M$, and thus $0\ne\af\subseteq\ann^r_\lf(M)$. But by hypothesis
the latter is zero which is a contradiction. 
\end{proof}

\noindent {\bf Remark 3.} Note that for non-semi-simple Lie algebras Corollary
\ref{cohnonsemisim} can also be obtained from \cite[Theorem 2.4]{F1} in
conjunction with \cite[Theorem 2.6]{FW}.
\vspace{.1cm}

\begin{cor}\label{whitehead}
Let $\lf$ be a left Leibniz algebra over a field of characteristic zero. If $M$
is a finite-dimensional right faithful irreducible $\lf$-bimodule, then $\HL^n
(\lf,M)=0$ for every non-negative integer $n$.
\end{cor}

\begin{proof}
If $\lf$ is semisimple, the assertion follows from \cite[Theorem 4.2]{FW}, and
if $\lf$ is not semisimple, then the assertion is a consequence of Corollary
\ref{cohnonsemisim}. 
\end{proof}

Finally, as a consequence of Corollary \ref{whitehead} and Corollary \ref{cohnonsemisim}
we obtain the Leibniz analogue of \cite[Theorem 2.4]{F1}:  

\begin{cor}\label{farnsteiner}
Let $\lf$ be a left Leibniz algebra over a field $\F$, and let $M$ be a finite-dimensional
right faithful irreducible $\lf$-bimodule such that $\HL^n(\lf,M)\ne 0$ for some non-negative
integer $n$. Then $\ch(\F)>0$ and $\lf$ is a semi-simple Lie algebra.
\end{cor}


\section{Cohomology of solvable Leibniz algebras}\label{solv}


As a special case of Corollary \ref{van} we obtain the following generalization
of the analogue of Barnes' vanishing theorem for finite-dimensional nilpotent Lie
algebras \cite[Lemma 3 or Theorem 1]{B1} to Leibniz algebras of arbitrary
dimension:

\begin{thm}\label{vannilp}
Let $\lf$ be a nilpotent left Leibniz algebra. If $M$ is a finite-dimensional
$\lf$-bimodule such that $M^\lf=0$, then $\HL^n(\lf,M)=0$ for every
non-negative integer $n$.
\end{thm}

\noindent Theorem \ref{vannilp} and Proposition \ref{nontriv} enable us
now to prove a Leibniz analogue of Dixmier's vanishing theorem for nilpotent
Lie algebras \cite[Th\'eor\`eme 1]{D} (see also \cite[Proposition 2.8]{FW}
for our preliminary attempt to obtain such a result).

\begin{thm}\label{dixmier}
Let $\lf$ be a finite-dimensional nilpotent left Leibniz algebra, and let $M$
be a finite-dimensional $\lf$-bimodule. If every composition factor of $M$
is non-trivial, then
\begin{eqnarray*}
\HL^n(\lf,M)\cong
\left\{
\begin{array}{cl}
M_0 & \mbox{if }n=0\\
0 & \mbox{if }n\ge 1\,.
\end{array}
\right.
\end{eqnarray*}
Moreover, if $M$ is symmetric, then $\HL^n(\lf,M)=0$ for every non-negative
integer $n$.
\end{thm}

\begin{proof}
For $n=0$ the assertion is just Proposition \ref{nontriv}. The proof for $n>0$
is divided into three steps. Firstly, for symmetric $\lf$-bimodules the statement
follows from Proposition \ref{nontriv} and Theorem \ref{vannilp}.

Now suppose that $M$ is anti-symmetric. It is clear that subbimodules and
homomorphic images of anti-symmetric bimodules are again anti-symmetric.
By using the long exact cohomology sequence, it is therefore enough to prove
the first part of the theorem for irreducible anti-symmetric $\lf$-bimodules.
In this case we obtain from \cite[Lemma 1.4\,(b)]{FW} that
$$\HL^n(\lf,M)\cong\HL^{n-1}(\lf,\Hom_\F(\lf,M)_s)\cong\HL^{n-1}(\lf,(\lf^*
\otimes M)_s)$$
for every positive integer $n$. By refining the left descending central series of $\lf$
(see \cite[Section 5]{F}), one can construct a composition series
$$\lf_{\ad,\ell}=\lf_k\supset\lf_{k-1}\supset\cdots\supset\lf_1\supset\lf_0=0$$
of the left adjoint $\lf$-module such that $\lf_j/\lf_{j-1}$ is the trivial one-dimensional
$\lf$-module $\F$ for every integer $1\le j\le k$. From the short exact sequences
$0\to\lf_{j-1}\to\lf_j\to\F\to 0$, we obtain by dualizing, tensoring each term with
$M$, and symmetrizing the short exact sequences:
$$0\to M_s\to(\lf_j^*\otimes M)_s\to(\lf_{j-1}^*\otimes M)_s\to 0$$
for every integer $1\le j\le k$. Since $M$ is a non-trivial irreducible left $\lf$-module,
we conclude that $M_s$ is a non-trivial irreducible symmetric $\lf$-bimodule. Hence,
we obtain inductively from the long exact cohomology sequence that $\HL^n(\lf,M)
\cong\HL^{n-1}(\lf,(\lf^*\otimes M)_s)=0$ for every positive integer $n$. 

Finally, if $M$ is arbitrary, then in the short exact sequence
$$0\to M_0\to M\to M_\sym\to 0$$
the first term is anti-symmetric and the third term is symmetric. Hence, another
application of the long exact cohomology sequence in conjunction with the
statements for the symmetric and the anti-symmetric case yields that $\HL^n
(\lf,M)=0$ for every positive integer.
\end{proof}

Let $\lf:=\F e$ be the one-dimensional Lie algebra. It is immediate from the
definition of Leibniz cohomology that $\dim_\F \HL^n(\lf,\F)=1$ for every
non-negative integer $n$. In the following we will generalize this to arbitrary
$\lf$-bimodules $M$. We will see that in positive degrees $\HL^n(\lf,M)$ is
periodic with period $2$. This is very similar to the cohomology of a finite
cyclic group.

\begin{thm}\label{ab}
Let $\lf:=\F e$ be the one-dimensional Lie algebra, and let $M$ be a
Leibniz $\lf$-bimodule. Then 
\begin{eqnarray*}
\HL^n(\lf,M)\cong
\left\{
\begin{array}{cl}
M^\lf & \mbox{if }n=0\\
M^0/M\lf & \mbox{if }n\mbox{ is odd}\\
M^\lf/M_0 & \mbox{if }n\mbox{ is even and }n\not=0
\end{array}
\right.
\end{eqnarray*}
{\rm (}\hspace{-.5mm}as
$\F$-vector spaces{\rm )} for every non-negative integer $n$, where
$$M^0:=\{m\in M\mid e\cdot m+m\cdot e=0\}\,.$$
Moreover, if $M$ is finite dimensional, then
$$M^0/M\lf\cong M^\lf/M_0$$
{\rm (}\hspace{-.5mm}as $\F$-vector spaces{\rm )}.
\end{thm}

\begin{proof}
As $\dim_\F\lf=1$, we have that $\CL^n(\lf,M):=\Hom_\F(\lf^{\otimes n},M)
\cong M$ for every integer $n\ge 0$. Moreover, we obtain that
$$\dl^n(m)=\sum_{j=1}^n(-1)^{j+1}e\cdot m+(-1)^{n+1}m\cdot e$$
for any $f\in\CL^n(\lf,M)$ from which it follows that
\begin{eqnarray*}
\dl^n(m)=
\left\{
\begin{array}{cl}
-m\cdot e & \mbox{if }n\mbox{ is even}\\
e\cdot m+m\cdot e & \mbox{if }n\mbox{ is odd}
\end{array}
\right.
\end{eqnarray*}
for every non-negative integer $n$, and therefore
\begin{eqnarray*}
\Ker(\dl^n)=
\left\{
\begin{array}{cl}
M^\lf & \mbox{if }n\mbox{ is even}\\
M^0 & \mbox{if }n\mbox{ is odd}
\end{array}
\right.
\end{eqnarray*}
and
\begin{eqnarray*}
\im(\dl^{n-1})=
\left\{
\begin{array}{cl}
M_0 & \mbox{if }n\mbox{ is even}\\
M\lf & \mbox{if }n\mbox{ is odd}\,.
\end{array}
\right.
\end{eqnarray*}
This immediately implies the first part of the theorem.

Now let us assume that $\dim_\F M<\infty$. If $(\lambda,\rho)$
denotes the associated representation of $M$, then we have that
$M^\lf=\Ker(\rho_e)$, $M\lf=\im(\rho_e)$, $M^0=\Ker(\lambda_e
+\rho_e)$, and $M_0=\im(\lambda_e+\rho_e)$. In this case we
deduce from the dimension formula for linear transformations that
$$\dim_\F M^\lf+\dim_\F M\lf=\dim_\F M=\dim_\F M^0+\dim_\F M_0\,,$$
or equivalently,
$$\dim_\F M^0/M\lf=\dim_\F M^\lf/M_0\,,$$
which finishes the proof of the second part of the theorem.
\end{proof}

Note that in the special case of the one-dimensional Lie algebra
$\lf=\F e$ the inclusions $M_0\subseteq M^\lf$ and $M\lf\subseteq M^0$
can be obtained directly without the coboundary property. Namely, it follows
from identity (\ref{MLLrep}) that $\rho_e\circ(\lambda_e+\rho_e)=0$,
and thus $M_0\subseteq M^\lf$. Similarly, we conclude from
(\ref{LMLrep}) and $e^2=0$ that $\lambda_e\circ\rho_e=\rho_e
\circ\lambda_e$, and therefore
$$(\lambda_e+\rho_e)\circ\rho_e=\lambda_e\circ\rho_e+\rho_e^2=
\rho_e\circ\lambda_e+\rho_e^2=\rho_e\circ(\lambda_e+\rho_e)=0\,,$$
which yields $M\lf\subseteq M^0$.
\vspace{.3cm}

It would be very interesting to find other Leibniz algebras with periodic
cohomology or even characterize all such Leibniz algebras. We hope
to come back to these questions on another occasion.

Next, we use Theorem \ref{ab} to prove a Leibniz analogue of
Dixmier's non-vanishing theorem for nilpotent Lie algebras
\cite[Th\'eor\`eme 2]{D}. According to Proposition \ref{nontriv},
$M^\lf\not=M_0$ implies that the $\lf$-bimodule $M$ has
a trivial composition factor. In the next result we prove that the
stronger of these two conditions is sufficient for the non-vanishing
of $\HL^n(\lf,M)$ in every positive degree. But note that it follows from
Theorem \ref{ab} that for the three-dimensional Leibniz bimodule
$M$ over the one-dimensional Lie algebra in Example A of Section
1 the cohomology vanishes in every positive degree, but every
composition factor of $M$ is trivial. This shows that the obvious
analogue of the hypothesis in \cite[Th\'eor\`eme 2]{D} is not strong
enough to guarantee the non-vanishing of Leibniz cohomology
in positive degrees.

\begin{thm}\label{nonvannilp}
Let $\lf$ be a non-zero finite-dimensional nilpotent left Leibniz algebra,
and let $M$ be a finite-dimensional $\lf$-bimodule. If $M^\lf\not=M_0$,
then $\HL^n(\lf,M)\ne 0$ for every non-negative integer $n$.
\end{thm}

\begin{proof}
For the proof we will use Dixmier's exact sequence (see \cite[Proposition
1]{D}), adapted to Leibniz algebras. Since $\lf$ is
nilpotent, it has an ideal $\iif$ of codimension one, and therefore $\lf
=\iif\oplus\F x$ for some element $x\in\lf\setminus\iif$. The restriction
of a Leibniz cochain of $\lf$ to the ideal $\iif$ induces a short exact
sequence of cochain complexes
$$0\to\DL^\bullet(\lf,M)\to\CL^\bullet(\lf,M)\stackrel{\res^\bullet}{\to}
\CL^\bullet(\iif,M)\to 0\,,$$
but the kernel cochain complex $\DL^\bullet(\lf,M):=\Ker(\res^\bullet)$ is
much more complicated than in the Lie algebra case because of
$$\DL^n(\lf,M)=\bigoplus_{j=1}^n\Hom_\F(\iif^{\otimes(j-1)}\otimes\F x
\otimes\iif^{\otimes(n-j)},M)\oplus\ldots\oplus\Hom_\F([\F x]^{\otimes n},M)\,,$$
where the dots indicate Hom spaces with two, three, \dots, $n-1$
occurrences of factors $\F x$ (at arbitrary places). By using Dixmier's method
and the Cartan relations for Leibniz cohomology, we can identify one summand
in the cohomology of $\DL^\bullet(\lf,M)$, namely, the summand corresponding to 
$$\DL^n_1(\lf,M):=\Hom_\F(\F x\otimes\iif^{\otimes(n-1)},M)\subseteq\CL^n
(\lf,M)\,.$$
Let us show that the isomorphism of vector spaces 
$$(-1)^n\varphi^n:\DL^n_1(\lf,M)\to\CL^{n-1}(\iif,M),\quad\varphi^n(f):=
\res^{n-1}(\iota_x^n(f))$$
(where $\iota_x^\bullet$ is the insertion operator into the first component) is
compatible with the Leibniz coboundary operator. Indeed, by identity (\ref{cartan}),
we have
$$(\varphi^{n+1}\circ\dl^n)(f)=\res^n(\iota_x^{n+1}(\dl^n f))=-\res^n(\dl^{n-1}
(i_x^n(f)))=-(\dl^{n-1}\circ\varphi^n)(f)$$
as $\res^n(\theta^n_x(f))=0$ since the first component of the cochain vanishes
on $\iif$. Note that this reasoning works for any integer $n\geq 1$ because of
the validity of the Cartan identity (\ref{cartan}) in this range. 

Furthermore, the composition of the connecting homomorphism
$$\partial^n:\HL^n(\iif,M)\to\HCE^{n+1}(\DL^\bullet(\lf,M),\dl^\bullet)$$
with $\varphi^{n+1}$ is given by $\res^n\circ\theta_x^n$. Indeed, $\partial^n$
is defined by lifting a cocycle $c\in\ZL^n(\iif,M)$ to $\tilde{c}$ in $\CL^n(\lf,M)$,
then taking $\dl^n(\tilde{c})$ and observing that the result lies in $\DL^{n+1}
(\lf,M)$. The composition with $\varphi^{n+1}$ gives thus a cocycle $c\in\ZL^n
(\iif,M)$
$$(\varphi^{n+1}\circ\partial^n)(c)=\res^n(\iota_x^{n+1}(\dl^n(\tilde{c})))
=(\res^n\circ\theta_x^n)(\tilde{c})$$
again because of the Cartan formula (\ref{cartan}) saying $\theta_x^n(\tilde{c})
=\iota_x^{n+1}(\dl^n(\tilde{c}))+\dl^{n-1}(\iota_x^n(\tilde{c}))$, where the
last term is a coboundary. 

Now let us prove the theorem. We proceed by induction on the dimension of $\lf$.
The base step follows from Theorem \ref{ab} in conjunction with the hypothesis.
Suppose therefore that for all nilpotent Leibniz algebras of dimension less than the
dimension of $\lf$ the cohomology with values in a finite-dimensional bimodule $M$
for which $M^\lf\ne M_0$ is non-zero. We apply the induction hypothesis
to an ideal $\iif$ in $\lf$ of codimension one. Note that $M^\kf\ne M_0(\kf)$ is satisfied
for every subalgebra $\kf$ of $\lf$, where $M_0(\kf):=\langle y\cdot m+m\cdot y
\mid y\in\kf,m\in M\rangle_\F$. Namely one obtains from $M^\kf=M_0(\kf)$, that
$M^\lf=M_0(\lf)$ as $M^\lf\subseteq M^\kf=M_0(\kf)\subseteq M_0(\lf)$, and the
other inclusion is always true as one can see from \cite[Lemma 3.7]{F}. Moreover,
the long exact sequence which we constructed in the beginning of the proof reads
$$\cdots\to\HL^n(\lf,M)\stackrel{\res^n}{\to}\HL^n(\iif,M)\stackrel{\partial^n}{\to}
\HCE^{n+1}(\DL^\bullet(\lf,M),\dl^\bullet)\stackrel{\sigma^{n+1}}{\to}\HL^{n+1}
(\lf,M)\to\cdots\,,$$
where the linear transformation $\sigma^n:\HCE^n(\DL^\bullet(\lf,M),\dl^\bullet)
\to\HL^n(\lf,M)$ is induced by the inclusion $\DL^n(\lf,M)\hookrightarrow\CL^n
(\lf,M)$. By the induction hypothesis and the preceding constructions, the connecting
homomorphism factors over
$$0\not=\HL^n(\iif,M)\stackrel{(\theta_x^n)_{\vert\iif}}\to\HL^n(\iif,M)\subseteq
\HCE^{n+1}(\DL^\bullet(\lf,M),\dl^\bullet)\,.$$
By virtue of Theorem \ref{fittinghh} and Remark 1, we may assume that $M=M_0
(\lf)$. Hence, $\lf$ acts locally nilpotently on $M$, and thus $x$ acts nilpotently on
$M$. From this one deduces that $\theta_x^n$ is nilpotent, and thus $\partial^n$
cannot be an isomorphism. More precisely, $\partial^n$ cannot be surjective onto
the factor $\DL^{n+1}_1(\lf,M)$. This entails that $\sigma^{n+1}\not=0$, and
therefore $\HL^{n+1}(\lf,M)\not=0$.
\end{proof}

As an immediate consequence of Theorem \ref{nonvannilp} we obtain the following result:

\begin{cor}\label{nonvantriv}
Let $\lf$ be a non-zero finite-dimensional nilpotent left Leibniz algebra. Then
$\HL^n(\lf,\F)\ne 0$ for every non-negative integer $n$.
\end{cor}

\noindent {\bf Remark 4.} Note that $\dim_\F\HL^n(\F e,\F)=1$ for every
non-negative integer $n$ shows that Corollary \ref{nonvantriv} (and thus
Theorem \ref{nonvannilp}) is best possible (see also Example C below for
a non-Lie Leibniz algebra with one-dimensional trivial cohomology in every
non-negative degree).
\vspace{.3cm}

Next, we apply Theorem \ref{nonvannilp} to prove that the cohomology of a
finite-dimensional nilpotent left Leibniz algebra with coefficients in the adjoint
bimodule does not vanish in any degree provided its left center is not contained
in its Leibniz kernel. Recall that 
$$C^\ell(\lf):=\{c\in\lf\mid\forall\,x\in\lf:cx=0\}$$
is the {\em left center\/} of a left Leibniz algebra $\lf$. It is well-known that
$\leib(\lf)\subseteq C^\ell(\lf)=(\lf_\ad)^\lf$ (see \cite[Proposition 2.13]{F}),
but not necessarily conversely.

\begin{cor}\label{adj}
Let $\lf$ be a finite-dimensional nilpotent left Leibniz algebra such that
$C^\ell(\lf)\ne\leib(\lf)$. Then $\HL^n(\lf,\lf_\ad)\ne 0$ for every
non-negative integer $n$.
\end{cor}

\begin{proof}
Note that $(\lf_\ad)^\lf=C^\ell(\lf)$, and by virtue of \cite[Example
3.11]{F}, we have $(\lf_\ad)_0\subseteq\leib(\lf)$. Consequently, we
obtain from $C^\ell(\lf)\ne\leib(\lf)$ that $(\lf_\ad)^\lf\ne(\lf_\ad)_0$,
and thus we can apply Theorem \ref{nonvannilp} to obtain the assertion.
\end{proof}


Since the center of a non-zero nilpotent Lie algebra is always non-zero,
we deduce from Corollary \ref{adj}:

\begin{cor}\label{adjlie}
If $\gf$ is a non-zero finite-dimensional nilpotent Lie algebra, then
$\HL^n(\gf,\gf_\ad)\ne 0$ for every non-negative integer $n$.
\end{cor}

Let us illustrate the above results by some cohomology computations for the
two-dimensional nilpotent non-Lie Leibniz algebra $\nf:=\F e\oplus\F f$ with
multiplication $ee=ef=fe=0$ and $ff=e$ (see \cite[Example 2.4]{F}).
\vspace{.2cm}

\noindent{\bf Example C:} In the following we consider the cohomology of
the two-dimensional nilpotent non-Lie Leibniz algebra $\nf$ with trivial or
adjoint coefficients.

In Example C of \cite{FW}, we stated incorrectly that the higher differentials
$d_r$ for $r\geq 2$ in Pirashvili's spectral sequence (see
\cite[Corollary 3.5]{FW}) are zero. But this is not the case. Indeed, the
differential $\dl_2:E_2^{0,1}\to E_2^{2,0}$ sends $e^*\in E_2^{0,1}=
\HL^0(\nf_\lie,\leib(\nf)^*_s)\otimes\HL^1(\nf,\F)$ to $\dl_2 e^*\in
E_2^{2,0}=\HL^2(\nf_\lie,\leib(\nf)^*_s)\otimes\HL^0(\nf,\F)$ with
$(\dl_2 e^*)(f\otimes f)=-e^*(ff)=-e^*(e)=-1$. 

By computing explicitly cocycles and coboundaries, one can show
$$\dim_\F\HL^0(\nf,\F)=\dim_\F\HL^1(\nf,\F)=\dim_\F\HL^2(\nf,\F)
=\dim_\F\HL^3(\nf,\F)=1\,.$$
More precisely, we have
$$\HL^0(\nf,\F)=\langle e\rangle_\F\,,$$
$$\HL^1(\nf,\F)=\langle f^*\rangle_\F\,,$$
$$\HL^2(\nf,\F))=\langle f^*\otimes e^*\rangle_\F\,,$$
$$\HL^3(\nf,\F)=\langle f^*\otimes e^*\otimes f^*\rangle_\F\,.$$
In fact, Gnedbaye has proven in \cite[(4.2), p.\ 22]{G} that
$\dim_\F\HL^n(\nf,\F)=1$ for every non-negative integer $n$.

Let us now discuss the adjoint cohomology of $\nf$. We have that
$$\HL^0(\nf,\nf_\ad)=C^\ell(\nf)=\F e\,.$$
Moreover, by computing explicitly cocycles and coboundaries, one can show
$$\dim_\F\HL^1(\nf,\nf_\ad)=\dim_\F\HL^2(\nf,\nf_\ad)=1\,.$$
We believe that this pattern continues in higher degrees, i.e., we conjecture 
that $\dim_\F\HL^n(\nf,\nf_\ad)=1$ for every non-negative integer $n$.

Note that in the case $\ch(\F)\ne 2$ we have that  $\nf_\ad^\nf=
\F e=(\nf_\ad)_0$. Our conjecture would imply that the condition in
Theorem \ref{nonvannilp} is not necessary for the non-vanishing of
$\HL^n(\nf,\nf_\ad)$ in any positive degree $n$. This example also
shows that for Leibniz algebras the condition in Theorem \ref{dixmier}
cannot be replaced by $M^\lf=M_0$, and thus the dichotomy in the
(non-)vanishing of their cohomology as present in Dixmier's theorems
does not seem to hold for Leibniz cohomology.
\vspace{.3cm}

We say that a Leibniz algebra $\lf$ is {\em supersolvable\/} if there
exists a chain
$$\lf=\lf_k\supset\lf_{k-1}\supset\cdots\supset\lf_1\supset\lf_0=0$$
of ideals of $\lf$ such that $\dim_\F\lf_j/\lf_{j-1}=1$ for every integer
$1\le j\le k$. Note that finite-dimensional nilpotent Leibniz algebras
are supersolvable and supersolvable Leibniz algebras are solvable.
Moreover, over algebraically closed fields of characteristic zero, every
finite-dimensional solvable Leibniz algebra is supersolvable (see
\cite[Corollary 2]{P2} or \cite[Corollary 6.7]{F}). Finally, as for Lie
algebras, it is not difficult to see that subalgebras and homomorphic
images of supersolvable Leibniz algebras are again supersolvable.

Since $\lf$ is supersolvable, we have that the derived subalgebra $\lf\lf$
of $\lf$ is a nilpotent ideal of $\lf$ (see the proof of \cite[Corollary 3]{P2}).
Hence, the following result is a special case of Corollary \ref{van}:

\begin{thm}\label{vansupsolv}
Let $\lf$ be a supersolvable left Leibniz algebra. If $M$ is a finite-dimensional
$\lf$-bimodule such that $M^{\lf\lf}=0$, then $\HL^n(\lf,M)=0$ for every
non-negative integer $n$.
\end{thm}

\noindent Similarly to the nilpotent case, Theorem \ref{vansupsolv} and
Proposition \ref{non1dim} enable us to prove a Leibniz analogue of Barnes'
vanishing theorem for supersolvable Lie algebras \cite[Theorem 3]{B1}
(see also \cite[Proposition 2.7]{FW}). More general than in Barnes' result,
we allow Leibniz bimodules that are not necessarily irreducible. Note also
that the proof of Theorem \ref{barnes} is very similar to the proof of
Theorem \ref{dixmier}.

\begin{thm}\label{barnes}
Let $\lf$ be a supersolvable left Leibniz algebra over an algebraically
closed field $\F$, and let $M$ be a finite-dimensional $\lf$-bimodule.
If no composition factor of $M$ is one-dimensional, then
\begin{eqnarray*}
\HL^n(\lf,M)\cong
\left\{
\begin{array}{cl}
M_0 & \mbox{if }n=0\\
0 & \mbox{if }n\ge 1\,.
\end{array}
\right.
\end{eqnarray*}
Moreover, if $M$ is symmetric, then $\HL^n(\lf,M)=0$ for
every non-negative integer $n$.
\end{thm}

\begin{proof}
For $n=0$ the assertion is just Proposition \ref{nontriv}. The proof for $n>0$
is divided into three steps. Firstly, for symmetric $\lf$-bimodules the statement
follows from Proposition \ref{non1dim} and Theorem \ref{vansupsolv}. 

Now suppose that $M$ is anti-symmetric. It is clear that subbimodules and
homomorphic images of anti-symmetric bimodules are again anti-symmetric.
By using the long exact cohomology sequence, it is therefore enough to prove
the first part of the theorem for irreducible anti-symmetric $\lf$-bimodules.
In this case we obtain from \cite[Lemma 1.4\,(b)]{FW} that
$$\HL^n(\lf,M)\cong\HL^{n-1}(\lf,\Hom_\F(\lf,M)_s)\cong\HL^{n-1}(\lf,
(\lf^*\otimes M)_s)$$
for every positive integer $n$. By definition of supersolvability, the left adjoint
$\lf$-module has a composition series
$$\lf_{\ad,\ell}=\lf_k\supset\lf_{k-1}\supset\cdots\supset\lf_1\supset\lf_0
=0$$
such that $\dim_\F\lf_j/\lf_{j-1}
=1$ for every integer $1\le j\le k$. From the short exact sequences $0\to
\lf_{j-1}\to\lf_j\to\lf_j/\lf_{j-1}\to 0$, we obtain by dualizing, tensoring each
term with $M$, and symmetrizing the short exact sequences:
$$0\to[(\lf_j/\lf_{j-1})^*\otimes M]_s\to(\lf_j^*\otimes M)_s\to(\lf_{j-1}^*
\otimes M)_s\to 0$$
for every integer $1\le j\le k$. Since $M$ is irreducible and $\dim_\F\lf_j/
\lf_{j-1}=1$, we conclude that $[(\lf_j/\lf_{j-1})^*\otimes M]_s$ is an
irreducible symmetric $\lf$-bimodule. Moreover, we have that $\dim_\F
[(\lf_j/\lf_{j-1})^*\otimes M]_s\ne 1$ as $\dim_\F M\ne 1$. Hence we
obtain inductively from the long exact cohomology sequence that $\HL^n
(\lf,M)\cong\HL^{n-1}(\lf,(\lf^*\otimes M)_s)=0$ for every positive integer
$n$.

The remainder of the proof is exactly the same as in the proof of
Theorem~\ref{dixmier}. 
\end{proof}

\noindent {\bf Remark 5.} According to Lie's theorem for Leibniz algebras,
every finite-dimen\-sional irreducible Leibniz bimodule of a finite-dimensional
solvable Leibniz algebra over an algebraically closed field of characteristic
zero is one-dimensional (see \cite[Theorem 2]{P2} or \cite[Corollary
6.5\,(a)]{F}). Consequently, in this case the hypothesis of Theorem
\ref{barnes} is never satisfied, and thus this result is only applicable
over fields of prime characteristic.
\vspace{.3cm}


Note that in addition to not knowing of a general condition that guarentees
the non-vanishing of the cohomology of a supersolvable Leibniz algebra in
all or at least in certain degrees, in many situations Theorem \ref{barnes}
cannot be applied. Let us illustrate this by some computations for the
cohomology of the two-dimensional supersolvable non-Lie Leibniz algebra
$\af=\F h\ltimes_\ell\F e$.
\vspace{.2cm}

\noindent {\bf Example D:} Let $\af=\F h\oplus\F e$ denote the
two-dimensional supersolvable non-Lie Leibniz algebra with multiplication
$he=e$ and $hh=eh=ee=0$ (see \cite[Example~2.3]{F}).

In \cite[Example D]{FW} we computed the cohomology of $\af$ with trivial
coefficients. Namely, we obtained that
$$\dim_\F\HL^n(\af,\F)=1$$
for every non-negative integer $n$.

For the adjoint cohomology of $\af$ we have that
$$\HL^0(\af,\af_\ad)=C^\ell(\af)=\F e\,.$$
Note that $\af_\ad$ has only one-dimensional composition factors, but by
computing explicitly cocycles and coboundaries, one can show that
$$\HL^1(\af,\af_\ad)=\HL^2(\af,\af_\ad)=0\,.$$

In particular, this shows that $\af$ has only inner derivations. Moreover,
it follows from \cite[Th\'eor\`eme 3]{Ba} that over an algebraically closed
field of characteristic zero $\af$ is a rigid Leibniz algebra, i.e., in this case
$\af$ has no non-trivial infinitesimal deformations.

We conjecture that the adjoint cohomology of $\af$ also vanishes in higher
degrees, namely, $\HL^n(\af,\af_\ad)=0$ for every positive integer $n$.

Let us remark that this vanishing behavior would be analogous to
\cite[Proposition~2.11]{FW}. As $\af=\F h\ltimes_\ell\F e$ is a
hemi-semidirect product and $\F h$ acts semisimply on $\F e$,
one might speculate whether the adjoint cohomology of any
hemi-semidirect product of an abelian Lie algebra that acts
semisimply on a nilpotent Lie algebra vanishes in every positive
degree.
\vspace{.3cm}

We conclude this section by generalizing Barnes' vanishing theorem for
solvable Lie algebras \cite[Theorem 2]{B1} to Leibniz algebras.

\begin{thm}\label{vansolv}
Let $\lf$ be a solvable left Leibniz algebra. If $M$ is a finite-dimensional
right faithful irreducible $\lf$-bimodule, then $\HL^n(\lf,M)=0$ for
every non-negative integer $n$.
\end{thm}

\begin{proof}
Suppose that $\lf\ne 0$ is semi-simple. Since by hypothesis $\lf $ is solvable,
this implies that $0\ne\lf=\leib(\lf)$ which contradicts \cite[Proposition 2.20]{F}.
As a consequence, we have that either $\lf=0$ or $\lf$ is not semisimple.
The first case contradicts the original assumption, and in the second case
the assertion follows from Corollary \ref{cohnonsemisim}.
\end{proof}


\section{Applications}\label{appl}


In \cite{B1} Barnes uses the cohomological vanishing theorems
proved in his paper to derive several structure theorems for
finite-dimensional nilpotent and (super)solvable Lie algebras.
The result for nilpotent Lie algebras \cite[Theorem~5]{B1}
has already been generalized and extended to Leibniz algebras
by Barnes himself (see \cite[Theorem~5.5]{B2}). In order to show
the usefulness of Theorem \ref{barnes}, we prove the analogue
of \cite[Theorem 6]{B1} for Leibniz algebras and extend two
characterizations of supersolvable Lie algebras (Theorems 7
and 8 in \cite{B1}) to Leibniz algebras. Recall that the {\em Frattini
subalgebra\/} $F(\lf)$ of a Leibniz algebra $\lf$ is the intersection
of the maximal subalgebras of $\lf$ (see \cite[Section 2]{T} or
\cite[Definition~5.4]{B2}). For the convenience of the reader we
include the details of the proofs.

Contrary to solvable Leibniz algebras, extensions of supersolvable
Leibniz algebras by supersolvable Leibniz algebras are not always
supersolvable. In certain situations the following result
can be used as a substitute.

\begin{thm}\label{frattini}
Let $\lf$ be a finite-dimensional left Leibniz algebra over an algebraically
closed field $\F$. If $\iif$ is an ideal of $\lf$ such that $\iif\subseteq
F(\lf)$ and $\lf/\iif$ is supersolvable, then $\lf$ is supersolvable. 
\end{thm}

\begin{proof}
If $\iif=0$, then the assertion is trivial. So suppose that $\iif\ne 0$
and proceed by induction on the dimension of $\lf$. For the base
step there is nothing to prove. Now choose a non-zero ideal $\af$
of $\lf$ of minimal dimension that is contained in $\iif$. It follows
from $\af\subseteq\iif\subseteq F(\lf)$ and \cite[Corollary 5.6]{B2}
that $\af$ is nilpotent. But since $\af$ has minimal dimension, we
then obtain that $\af$ is abelian, and therefore $\af$ is an irreducible
$\lf/\af$-module. From $\af\subseteq\iif\subseteq F(\lf)$ and
\cite[Proposition 4.3\,(ii)]{T} we deduce that $\iif/\af\subseteq
F(\lf)/\af=F(\lf/\af)$. On the other hand, we have that
$(\lf/\af)/(\iif/\af)\cong\lf/\iif$ is supersolvable,
and thus the induction hypothesis yields that $\lf/\af$ is also
supersolvable.

Suppose now that $\dim_\F\af>1$. In this case we conclude from
Theorem \ref{barnes} that the extension of $\af$ by $\lf/\af$ splits
(see \cite[Section 1.7]{LP} or \cite[Theorem 1.3.13]{C}), and thus
there exists a subalgebra $\hf$ of $\lf$ such that $\lf=\af\oplus\hf$.
But then it follows from $\af\subseteq\iif\subseteq F(\lf)$ that $\lf=
F(\lf)+\hf$, and we conclude from \cite[Lemma~2.1]{T} that
$\lf=\hf$. Consequently, we have that $\af=\af\cap\lf=\af\cap
\hf=0$, which is a contradiction. Hence, we obtain that $\af$ is
one-dimensional.

Since $\lf/\af$ is supersolvable, there exists a chain of ideals
$$\lf/\af=L_k\supset L_{k-1}\supset\cdots\supset L_1\supset L_0=0$$
such that $\dim_\F L_j/L_{j-1}=1$ for every integer $1\le j\le k$.
Then there exists ideals $\lf_j$ of $\lf$ such that $\af\subseteq\lf_j$
and $L_j=\lf_{j+1}/\af$ for every integer $0\le j\le k$. Finally,
the chain of ideals
$$\lf=\lf_{k+1}\supset\lf_k\supset\cdots\supset\lf_2\supset\lf_1=
\af\supset\lf_0=0$$
with $\dim_\F\lf_{j+1}/\lf_j=\dim_\F L_j/L_{j-1}=1$ for every integer
$1\le j\le k$ shows that $\lf$ is supersolvable.
\end{proof}

As an application of Theorem \ref{frattini} we obtain the following
characterization of supersolvable Leibniz algebras in terms of their
maximal subalgebras:

\begin{cor}\label{max}
Let $\lf$ be a finite-dimensional left Leibniz algebra over an
algebraically closed field $\F$. Then $\lf$ is supersolvable if,
and only if, $\lf$ is solvable and every maximal subalgebra
of $\lf$ has codimension one.
\end{cor}

\begin{proof}
Suppose that $\lf$ is supersolvable. We proceed by induction on
$\dim_\F\lf$. Let $\mf$ be a maximal subalgebra. Choose a minimal
ideal $\af$ of $\lf$. If $\mf\supseteq\af$, then $\mf/\af$ is a maximal
subalgebra of $\lf/\af$, and therefore it follows from the induction
hypothesis that $\dim_\F(\lf/\af)/(\mf/\af)=1$. Hence, we obtain
that $\lf/\mf\cong(\lf/\af)/(\mf/\af)$ is one-dimensional. On the
other hand, if $\mf\not\supseteq\af$, then $\lf=\mf+\af$. Since
$\lf$ is supersolvable, $\af$ is one-dimensional, and thus $\mf
\cap\af=0$. Consequently, we obtain that $\lf=\mf\oplus\af$,
which implies that $\dim_\F\lf/\mf=\dim_\F\af=1$.

Conversely, suppose that $\lf$ is solvable and every maximal
subalgebra of $\lf$ has codimension one. We again proceed
by induction on $\dim_\F\lf$. Let $\af$ be a minimal ideal of
$\lf$. Then it follows from the induction hypothesis that $\lf/
\af$ is supersolvable. If $\af\subseteq F(\lf)$, we conclude
from Theorem \ref{frattini} that $\lf$ is supersolvable.
Otherwise, if $\af\not\subseteq F(\lf)$, then there exists a
maximal subalgebra $\mf$ that does not contain $\af$, and
therefore $\lf=\mf+\af$. Since by hypothesis $\lf$ is solvable,
$\af$ is abelian. Hence, $\mf\cap\af$ is an ideal of $\lf$, and
thus $\af\cap\mf=0$. Consequently, we obtain that $\lf=\mf
\oplus\af$, which yields $\dim_\F\af=\dim_\F\lf/\mf=1$.
The rest of the proof is then exactly the same as in the proof
of Theorem \ref{frattini}.  
\end{proof}

\noindent {\bf Remark 6:} After finishing our paper we became
aware of the paper \cite{B3}. In Corollary 3.10 of this paper
Barnes proves Corollary \ref{max} for arbitrary fields by using
the theory of formations and projectors.
\vspace{.3cm} 

As a consequence of Corollary \ref{max}, we can deduce the
following lattice-theoretic characterization of supersolvable
Leibniz algebras:

\begin{cor}\label{maxchain}
Let $\lf$ be a finite-dimensional left Leibniz algebra over an
algebraically closed field $\F$. Then $\lf$ is supersolvable
if, and only if, $\lf$ is solvable and all maximal chains of
subalgebras of $\lf$ have the same length.
\end{cor}

\begin{proof}
Suppose that $\lf$ is supersolvable. According to Corollary
\ref{max}, every subalgebra in a maximal chain of subalgebras
of $\lf$ has codimension one, and therefore the length of such
a chain is $\dim_\F\lf$.

Conversely, suppose that $\lf$ is solvable and all maximal chains
of subalgebras of $\lf$ have the same length. Then there exists a
chain
$$\lf=\hf_d\supset\hf_{d-1}\supset\cdots\supset\hf_1\supset\hf_0
=0$$
of subalgebras of $\lf$ such that $\hf_{i-1}$ is a maximal ideal of
$\hf_i$ (but not necessarily an ideal of $\lf$). Since $\lf$ is solvable,
$\dim_\F\hf_i/\hf_{i-1}=1$ for every integer $1\le i\le d$. In particular,
this chain of subalgebras has length $\dim_\F\lf$, and therefore by
hypothesis every maximal chain of subalgebras of $\lf$ has length
$\dim_\F\lf$.

Now let $\mf$ be a maximal subalgebra of $\lf$. By successively
choosing maximal subalgebras we obtain a chain
$$\lf=\mf_r\supset\mf=\mf_{r-1}\supset\cdots\supset\mf_1\supset
\mf_0=0$$
of subalgebras of $\lf$ such that $\mf_{j-1}$ is maximal in $\mf_j$
for every integer $1\le j\le r$. This chain is clearly maximal, and thus
we have that $r=\dim_\F\lf$, or equivalently, $\dim_\F\mf_j/\mf_{j-1}
=1$ for every integer $1\le j\le r$. In particular, we obtain that $\dim_\F
\lf/\mf=1$, and then the assertion follows from Corollary \ref{max}.
\end{proof}

\noindent {\bf Remark 7:} After finishing our paper we became
aware of the paper \cite{ST}. The equivalence of the statements
(ii) and (iii) in Proposition 5.1 of this paper is closely related to
our Corollary \ref{maxchain}.
\vspace{.3cm} 

Let $\lf$ be a Leibniz algebra, and let $S$ be any subset of $\lf$. Then
$$C_\lf^r(S):=\{x\in\lf\mid\forall\,s\in S:sx=0\}$$
denotes the {\em right centralizer\/} of $S$ in $\lf$. We conclude this
section by extending \cite[Theorem 4]{B1} from Lie algebras to Leibniz
algebras. If $D$ is a derivation of an algebra $\rf$ over the real or complex
numbers, then
$$\exp(D):=\id_\rf+D+\frac{1}{2}D^2+\frac{1}{3!}D^3+\cdots$$
is an automorphism of $\rf$. The same is true for a nilpotent derivation
$D$ of an algebra $\rf$ over an arbitrary field of characteristic zero
(see \cite[Section 2.3]{H}). Note that we do not need to assume that
the characteristic of the ground field of $\rf$ is zero if $D^2=0$. We
say that two subalgebras $\kf$ and $\hf$ of a left Leibniz algebra $\lf$
are {\em conjugate\/} if there exists an element $x\in\lf$ such that
$\exp(L_x)(\kf)=\hf$, where $L_x$ denotes the left multiplication
operator of $x$ on $\lf$ satisfying the appropriate nilpotency condition
depending on the ground field of $\lf$.

Note that the Leibniz algebra in the next result does not have to be finite
dimensional as in Barnes' result \cite[Theorem 4]{B1}.

\begin{thm}\label{splitsolv}
Let $\lf$ be a solvable left Leibniz algebra. If $\af$ is a finite-dimensional
minimal ideal of $\lf$ such that $C_\lf^r(\af)=\af$, then every extension
of $\af$ by $\lf/\af$ splits and all complements of $\af$ in
$\lf/\af$ are conjugate.
\end{thm}

\begin{proof}
Since every minimal ideal of a solvable Leibniz algebra is abelian, $\af$ is
abelian, and therefore $\af$ is an $\lf/\af$-bimodule via the action induced
by left and right multiplication on $\lf$. The hypothesis that $\af$ is right
self-centralizing implies that $\af$ is a right faithful $\lf$-bimodule. Finally,
we obtain from the minimality of $\af$ that $\af$ is an irreducible $\lf/
\af$-bimodule.
Hence, it follows from Theorem \ref{vansolv} that $\HL^2(\lf/\af,\af)=0$,
and therefore every extension of $\af$ by $\lf/\af$ splits (see \cite[Section~1.7]{LP}
and \cite[Theorem 1.3.13]{C}).

Now let $\kf$ and $\kf^\prime$ be two complements of $\af$ in $\lf/\af$,
i.e., $\kf$ and $\kf^\prime$ are subalgebras of $\lf$ such that $\lf=\af
\oplus\kf$ and $\lf=\af\oplus\kf^\prime$, respectively. Let $x\in\lf$ be
arbitrary. Then $x=a+k$ for some uniquely determined elements $a\in
\af$ and $k\in\kf$. Similarly, $x=a^\prime+k^\prime$ for some uniquely
determined elements $a^\prime\in\af$ and $k^\prime\in\kf^\prime$.
Then the linear transformation $D:\lf/\af\to\af$, $x+\af\mapsto k-k^\prime$
is well-defined. Namely, note that $k$ and $k^\prime$ do not change if $x$
is replaced by $x+a_0$ for some $a_0\in\af$. Moreover, let $\pi:\lf\to\lf/\af$
denote the natural epimorphism of Leibniz algebras. The computation $\pi
(k-k^\prime)=\pi(k)-\pi(k^\prime)=\pi(x)-\pi(x)=0$ shows that $k-k^\prime
\in\Ker(\pi)=\af$.

Next, we prove that $D$ is a derivation. For any two elements $x,y\in\lf$
there exist unique elements $k_x,k_y\in\kf$ and $a_x,a_y\in\af$ such that
$x=a_x+k_x$ and $y=a_y+k_y$. Since $\af$ is abelian, we have that
$$xy=(a_x+k_x)(a_y+k_y)=a_xk_y+k_xa_y+k_xk_y\,.$$
Similarly, there exist unique elements $k_x^\prime,k_y^\prime\in\kf^\prime$
and $a_x^\prime,a_y^\prime\in\af$ such that $x=a_x^\prime+k_x^\prime$
and $y=a_y^\prime+k_y^\prime$, and we obtain that
$$xy=(a_x^\prime+k_x^\prime)(a_y^\prime+k_y^\prime)=a_x^\prime
k_y^\prime+k_x^\prime a_y^\prime+k_x^\prime k_y^\prime\,.$$
From this we conclude that $D(xy+\af)=k_xk_y-k_x^\prime k_y^\prime$,
and we compute that
\begin{eqnarray*}
D[(x+\af)(y+\af)] & = & D(xy+\af)=k_xk_y-k_x^\prime k_y^\prime\\
& = & (k_x-k_x^\prime)k_y+k_x^\prime(k_y-k_y^\prime)\\
& = & (k_x-k_x^\prime)\cdot(k_y+\af)+(k_x^\prime+\af)\cdot(k_y-k_y^\prime)\\
& = & (k_x-k_x^\prime)\cdot(y+\af)+(x+\af)\cdot(k_y-k_y^\prime)\\
&= & D(x+\af)\cdot(y+\af)+(x+\af)\cdot D(y+\af)\,.
\end{eqnarray*}
It follows from Theorem \ref{vansolv} that $\HL^1(\lf/\af,\af)=0$, and thus
we obtain from \cite[Proposition 4.3]{F} that there exists an element $a\in\af$
such that $D(x+\af)=-ax$ for any element $x\in\lf$. Since $\af$ is abelian,
we have that $L_a^2=0$, and therefore $\sigma:=\id_K+L_a$ defines an
automorphism of $\lf$ (see the argument in \cite[Section~2.3]{H}).

For any element $k\in\kf$ we have that $k=0+k=a^\prime+k^\prime$ for some
elements $a^\prime\in\af$ and $k^\prime\in\kf^\prime$. In particular, we obtain
that $D(k+\af)=k-k^\prime$, and thus
$$\sigma(k)=k+ak=k-D(k+\af)=k-(k-k^\prime)=k^\prime\in\kf^\prime\,,$$
which shows that $\sigma(\kf)\subseteq\kf^\prime$. But as an automorphism
$\sigma$ is injective. Hence, the restriction $\sigma_{\vert\kf}$ of $\sigma$ to
$\kf$ is also injective. Now it follows from $\dim_\F\kf=\dim_\F\lf/\af=\dim_\F
\kf^\prime$ that $\sigma_{\vert\kf}:\kf\to\kf^\prime$ is surjective, i.e., $\sigma
(\kf)=\kf^\prime$.
\end{proof}


\noindent {\bf Acknowledgments.} Most of this paper was written during a
visit of the second author to the University of South Alabama in May 2022.
Both authors are very grateful to the University of South Alabama and the
Universit\'e de Nantes for the financial support.



\end{document}